\documentclass[10pt,a4paper]{article}
\usepackage{latexsym,amsthm,amssymb,amscd,amsmath}
\newcounter{alphthm}
\theoremstyle{plain}
\newtheorem{theorem}{Theorem}[section]
\newtheorem{lemma}[theorem]{Lemma}
\newtheorem{proposition}[theorem]{Proposition}
\newtheorem{cor}[theorem]{Corollary}
\theoremstyle{definition}

\newtheorem{rem}[theorem]{Remark}
\newtheorem{example}[theorem]{Example}

\newcommand{\be}{\begin{equation}}
\newcommand{\ee}{\end{equation}}
\newcommand{\ben}{\begin{enumerate}}
\newcommand{\een}{\end{enumerate}}

\begin{document}
\title{Some Results Related to Soft Topological Spaces}
\author{E. Peyghan , B. Samadi and A. Tayebi}
\maketitle

\maketitle
\begin{abstract}
The notion of soft sets is introduced as a general mathematical tool for dealing with
uncertainty. In this paper,  we consider the concepts of soft compactness, countably soft compactness and obtain some results. We study some soft separation axioms that have been studied by  Min and Shabir-Naz. By constructing a special soft topological space, show that some classical results in general topology are not true about soft topological spaces, for instance every compact Housdorff spaces need not be normal.
\end{abstract}

\section{Introduction}

During recent years General Topology was developed by many mathematicians. The theory of generalized topological spaces, which was founded by \'{A}. Cs\'{a}sz\'{a}r is one of these developments \cite{CS}. Recently, in \cite{SN} Shabir-Naz introduced and studied the concepts of soft topological spaces and some related concepts. The generalized topology is different from topology by its axioms ( A collection of subsets of X is a generalized topology on X if and only if it contains empty set and arbitrary union of its elements). But the soft topology is based on soft sets theory and not sets.

Some notions in Mathematics can be considered as mathematical tools for dealing with uncertainties, namely theory of fuzzy sets, theory of intuitionistic fuzzy sets,   theory of vague sets, theory of rough sets and etc. But all of these theories have their own difficulties. In \cite{M}, Molodtsov introduced the concept of a soft set in order to solve complicated problems in the economics, engineering, and environmental areas because no mathematical tools can successfully deal with the various kinds of uncertainties
in these problems.  He successfully applied the soft theory in several directions, such as  game theory, probability, Perron integration, Riemann integration and  theory of measurement \cite{M, MLK}.

In \cite{MBR}, Maji-Biswas-Roy defined and studied operations of soft sets. Then  Pei-Miao \cite{PM}  and Chen \cite{C} improved the work of Maji-Biswas-Roy \cite{MBR0, MBR}. The properties and applications of soft set theory have been studied increasingly in \cite{Ali}. In \cite{CE}, \c{C}a\u{g}man-Enginoglu redefined the operations of the soft sets  and constructed a uniint decision making method  by using these new operations, and developed soft set theory. Then to make easy compaction with the operations of soft sets, they presented the soft matrix theory and set up the soft max–min decision making method \cite{CE2}. These decision making methods  can be successfully applied to many problems that contain uncertainties. In \cite{SN}, the authors studied some concepts related to soft spaces such as soft interior,  soft subspace and soft separation axioms. Recently,   Aygunoglu-Aygun introduced the soft product topology and defined the version of compactness in soft spaces named soft compactness \cite{AA}.

In this paper,  we consider the concepts of soft compactness and countably soft compact and get some results. Then, we study some soft separation axioms that have studied by  Min and Shabir-Naz. By constructing some examples we show that some classical results in general topology are not true about soft topological spaces, for instance every compact Housdorff spaces need not be normal.

\section{Preliminaries}
In this section, we recall some definitions and concepts discussed
in \cite{HA, WKM, SN, ZAMA}. Let $U$ be an initial universe and
$E$ be a set of parameters. Let $\mathbb{P}(U)$ denotes the power
set of $U$ and $A$ be a nonempty subset of $E$. A pair $(F, A)$ is
called a {\it soft set} over $U$, where $F$ is a mapping given by
$F: A\rightarrow\mathbb{P}(U)$. For two soft sets $(F, A)$ and
$(G, B)$ over common universe $U$, we say that $(F, A)$ is a {\it
soft subset} $(G, B)$ if $A\subseteq B$ and $F(e)\subseteq G(e)$,
for all $e\in A$. In this case, we write $(F,
A)\widetilde{\subseteq}(G, B)$ and $(G, B)$ is said to be a {\it
soft super set} of $(F, A)$. Two soft sets $(F, A)$ and $(G, B)$
over a common universe $U$ are said to be {\it soft equal} if $(F,
A)\widetilde{\subseteq}(G, B)$ and $(G, B)\widetilde{\subseteq}(F,
A)$. A soft set $(F, A)$ over $U$ is called a {\it null soft set},
denoted by $\Phi_A$, if for each $e\in A$, $F(e)=\emptyset$.
Similarly, it is called {\it absolute soft set}, denoted by
$\widetilde{U}$, if for each $e\in A$, $F(e)=U$.

\bigskip

The {\it union} of two soft sets $(F, A)$ and $(G, B)$ over the common universe
$U$ is the soft set $(H, C)$, where $C=A\cup B$ and for each $e\in
C$,
\[
H(e)=\left\{
\begin{array}{ccc}
F(e)&e\in A-B\\
G(e)&e\in B-A\\
F(e)\cup G(e)&e\in A\cap B
\end{array}
\right.
\]
We write $(F, A)\cup(G, B)=(H, C)$. Moreover, the
{\it intersection} $(H, C)$ of two soft sets $(F, A)$ and $(G,
B)$ over a common universe $U$, denoted by $(F, A)\cap(G, B)$, is
defined as $C=A\cap B$ and $H(e)=F(e)\cap G(e)$ for each $e\in C$. The {\it difference} $(H, E)$ of two soft sets $(F, E)$ and $(G, E)$ over $X$, denoted by $(F, E)\backslash(G, E)$, is defined
as $H(e)=F(e)\backslash G(e)$, for each $e\in E$. Let $Y$ be a nonempty subset of $X$. Then $\widetilde{Y}$ denotes
the soft set $(Y, E)$ over $X$ where $Y(e)=Y$, for each $e\in
E$. In particular, $(X, E)$ will be denoted by $\widetilde{X}$. Let $(F, E)$ be a soft set over $X$ and $x\in X$. We say that
$x\in(F, E)$, whenever $x\in F(e)$, for each $e\in E$ \cite{PST}.

\bigskip

The relative complement of a soft set $(F, A)$ is denoted by $(F,
A)'$ and is defined by $(F, A)'=(F', A)$ where $F: A\rightarrow
\mathbb{P}(U)$ is defined by following
\[
F'(e)=U-F(e), \ \ \ \forall e\in A.
\]

\bigskip

Let $\tau$ be the collection of soft sets over $X$. Then $\tau$ is called a soft topology on $X$ if $\tau$ satisfies the following axioms:

(i)\ $\Phi_E$, $\widetilde{X}$ belong to $\tau$.

(ii)\ The union of any number of soft sets in $\tau$ belongs to $\tau$.

(iii)\ The intersection of any two soft sets in $\tau$ belongs to $\tau$.\\
The triple $(X, \tau, E)$ is called a soft topological space over
$X$. The members of $\tau$ are said to be soft open in $X$, and the
soft set $(F, E)$ is called soft closed in $X$ if its relative
component $(F, E)'$ belongs to $\tau$.

The proof of the following proposition is an easy application of
De Morgan's lows  with the definition of a soft topology on $X$ (see Proposition 3.3 of \cite{ZAMA}).
\begin{proposition}\label{prop21}
\emph{Let $(X, \tau, E)$ be a soft space over $X$. Then}
\begin{description}
  \item[1)] \ \emph{$\Phi_E$, $\widetilde{X}$ are closed soft set over $X$;}
  \item[2)]  \emph{The intersection of any number of soft closed sets is a soft closed set over $X$;}
  \item[3)]  \emph{The union of any two soft closed sets is a soft closed set over $X$.}
\end{description}
\end{proposition}
\section{Soft Compactness}
In this section, we are going to introduce the concept of soft
compactness about soft topological spaces and study some properties related to these spaces (also, see \cite{ZAMA}).

A family $\mathcal{A}=\{(F_\alpha, E)\}_{\alpha\in J}$  of soft sets is a cover of a soft set $ (F,E) $ if
\[
(F, E)\widetilde{\subseteq}\bigcup_{\alpha\in J}(F_\alpha, E).
\]
It is a soft open cover if each member of $\mathcal{A}$ is a soft open set. A subcover of $\mathcal{A}$ is a subfamily of $\mathcal{A}$ which is also a cover. A soft topological space $(X, \tau, E)$ is said to be soft compact if each soft open cover of $(X, E)$ has a finite subcover.

Let $(X, \tau_1, E)$ and $(X, \tau_2, E)$ be soft topological spaces. If $\tau_1\subseteq\tau_2$, then $\tau_2$ is soft finer than $\tau_1$. If $\tau_1\subseteq\tau_2$ or $\tau_2\subseteq\tau_1$, then $\tau_1$ is soft comparable with $\tau_2$. Then, we have the following.
\begin{proposition}
Let $(X, \tau_2, E)$ be a soft compact space and $\tau_1\subseteq\tau_2$. Then $(X, \tau_1, E)$ is soft compact.
\end{proposition}
\begin{proof}
Let $\{(F_\alpha, E)\}_{\alpha\in J}$ be a soft open cover of $\widetilde{X}$ by soft open sets of $(X, \tau_1, E)$. Since $\tau_1\subseteq\tau_2$, then $\{(F_\alpha, E)\}_{\alpha\in J}$ is a soft open cover of $\widetilde{X}$ by soft open sets of $(X, \tau_2, E)$. But $(X, \tau_2, E)$ is soft compact. Therefore
\[
(X, E)\widetilde{\subseteq}(F_{\alpha_1}, E)\cup\ldots\cup(F_{\alpha_n}, E),
\] for some $\alpha_1,\ldots,\alpha_n\in J$. Hence $(X, \tau_1, E)$ is soft compact.
\end{proof}
In this paper, for convenience, let $SS(X)_E$ be the family of soft sets over $X$ with set of parameters $E$. We will apply two next propositions so much in the proofs.
\begin{proposition}\label{3.2}
Let $(F, E)$, $(G, E)$, $(H, E)$ and $(I, E)$ be soft sets in $SS(X)_E$. Then the following hold.
\begin{description}
  \item[(i)] $(F, E)\widetilde{\subseteq}(G, E)$ if and only if $(F, E)\cap(G, E)=(F, E)$;
  \item[(ii)] $(F, E)\widetilde{\subseteq}(G, E), (H, E)$ if and only if $(F, E)\widetilde{\subseteq}(G, E)\cap(H, E)$;
      \item[(iii)] If $(F, E)\widetilde{\subseteq}(H, E)$ and $(G, E)\widetilde{\subseteq}(I, E)$, then $(F, E)\cup(G, E)\widetilde{\subseteq}(H, E)\cup(I, E)$;
      \item[(iv)] $(F, E)\cap(F, E)'=\Phi_E$;
      \item[(v)] $(F, E)\cap(G, E)=\Phi_E$ if and only if $(F, E)\widetilde{\subseteq}(G, E)'$;
      \item[(vi)]  $(F,E)\widetilde{\subseteq}(G,E)$ if and only if $(G,E)'\widetilde{\subseteq}(F,E)'$.
\end{description}
\end{proposition}
\begin{proof}
Here, we only prove the (iii). Let $ (F,E)\cup (G,E)=(J,E)$ and $(H,E)\cup (I,E)=(K,E)$. Since $(F,E)\widetilde{\subseteq}(H,E)$ and $(G,E)\widetilde{\subseteq}(I,E) $; then
\[
F(e){\subseteq}H(e) \ \   and \ \ G(e){\subseteq}I(e), \ \ \forall e\in E.
\]
Therefore
\[
J(e)=F(e)\cup G(e)\subseteq H(e)\cup I(e)=K(e).
\]
Hence $(J,E)\widetilde{\subseteq}(K,E)$.
\end{proof}
Also we can obtain the following easily.
\begin{proposition}\label{3.3}
Let $(F,E)$ be a soft set and $\{(F_{\alpha},E)\}_{\alpha\in J} $ be a family of soft sets in $SS(X)_E$. Then the following hold.
\begin{description}
  \item[(i)] $(F, E)\cap(F, E)'=\Phi_E$;
  \item[(ii)] $(F, E)\cup\Phi_E=(F, E)$;
  \item[(iii)] $(F, E)\cap (\cup_{\alpha\in J}(F_{\alpha},E))=\cup_{\alpha\in J}((F,E)\cap(F_{\alpha},E))$;
  \item[(iv)] $\Phi_E'=\widetilde{X}$;
  \item[(v)] $\widetilde{X}'=\Phi_E$.
\end{description}
\end{proposition}

\bigskip

Let $(F,E)$ be a soft set over $X$ and $Y$ be a nonempty subset of $X$. Then the sub-soft set of $(F,E)$ over $Y$ denoted by $(^YF,E)$ is defined as follows
\[
^YF(e)=Y\cap F(e),
\]
for each $e\in E$. In other word $(^YF, E)=\widetilde{Y}\cap(F, E)$. Now, suppose that $(X, \tau, E)$  be a soft topological space over $X$ and $Y$ be a nonempty subset of $X$. Then
\[
\tau_Y=\{(^YF, E)|(F, E)\in\tau\},
\]
is said to be soft relative topology on $Y$ and $(Y, \tau_Y, E)$ is called a soft subspace of $(X, \tau, E)$. Here, we exhibit a criterion that applies $\widetilde{Y}$ is soft compact by soft open covers  of $\widetilde{Y}$, that all of members are soft open sets in $X$.

\begin{theorem}\label{babi}
Let $(Y,\tau_Y, E)$ be a soft subspace of a soft space $(X,\tau,E)$. Then $(Y,\tau_Y, E)$ is soft compact if and only if every cover of $\widetilde{Y}$ by soft open sets in $X$ contains a finite subcover.
\end{theorem}
\begin{proof}
Let $(Y,\tau_Y, E)$ be soft compact and $\{(F_\alpha,E)\}_{\alpha\in J}$ be a cover of $\widetilde{Y}$ by soft open sets in $X$. By Propositions \ref{3.2} and  \ref{3.3},  we can see that
$\{^YF_\alpha,E\}_{\alpha\in J}$ is a soft open cover of $\widetilde{Y}$. Therefore
\[
(Y,E)\widetilde{\subseteq}(^YF_{\alpha_1},E)\cup\ldots\cup (^YF_{\alpha_n},E),
\]
for some $\alpha_1,\ldots,\alpha_n\in J$. This implies that $\{(F_{\alpha_i},E)\}_{i=1}^n$ is a subcover of $\widetilde{Y}$ by soft open sets in $X$.
Conversely, let $\{(^YF_\alpha,E)\}_{\alpha\in J}$ be a soft open cover of $\widetilde{Y}$. It is easy to see that $\{(F_\alpha,E)\}_{\alpha\in J}$ is a cover of $\widetilde{Y}$ by soft open sets in $X$. Then we can write
\[
\widetilde{Y}\widetilde{\subseteq}(F_{\alpha_1},E)\cup,\ldots,\cup(F_{\alpha_n},E),
\]
for some $\alpha_1,\ldots,\alpha_n\in J$. Therefore $\{(^YF_{\alpha_i},E)\}_{i=1}^n$ is a subcover of $\widetilde{Y}$. Hence $(Y, \tau_Y, E)$ is soft compact.
\end{proof}
\begin{theorem}
Every soft compact subspace of a soft Hausdorff space is soft closed.
\end{theorem}
\begin{proof}
Let $(Y, \tau_Y, E)$ be a soft compact subspace of soft Hausdorff space $(X, \tau, E)$. Let $x\in(X, E)-(Y, E)$. Then for all $y\in(Y, E)$, $x\neq y$. Therefore, there exist soft open sets $(U_y, E)$ and $(U_{xy}, E)$ containing $x$ and $y$, respectively such that $(U_y, E)\cap(U_{xy}, E)=\Phi_E$. Obviously, $\{(U_{xy}, E)\}_{y\in Y}$ is a cover of $\widetilde{Y}$ by soft open sets in $X$. By Theorem \ref{babi}, we have $(Y, E)=(U_{xy_1}, E)\cup\ldots\cup(U_{xy_n}, E)$ for some $y_1,\ldots,y_n\in Y$. Now, $x\in(U_{y_1}, E)\cap\ldots\cap(U_{y_n}, E)=(U_x, E)$ and Proposition \ref{3.3} implies that $(U_x, E)\cap(Y, E)=\Phi_E$. Hence $x\in(U_x, E)\subseteq(X, E)-(Y, E)$. Then $(X, E)-(Y, E)=\bigcup_{x\in X-Y}(U_x, E)$. Therefore $(X, E)-(Y, E)$ is soft open. Hence $(Y, E)$ is soft closed.
\end{proof}

\bigskip

Using  Propositions \ref{3.2} and  \ref{3.3}, we are going to prove that every soft closed subspace of a soft compact space is soft compact.

\begin{theorem}
Every soft closed subset of a soft compact space is soft compact.
\end{theorem}
\begin{proof}
Let $(Y,\tau_Y, E)$ be a soft subspace of a soft compact space $(X, \tau, E)$ such that $(Y,E)$ is a soft closed in $X$. Let $\{(F_\alpha,E)\}_{\alpha\in J}$ be a cover of $\widetilde{Y}$ by soft open sets in $X$. $(Y,E)'$ is a soft open set in $X$. Propositions \ref{3.2} and  \ref{3.3} show that $\{(F_\alpha, E)\}_{\alpha\in J}\cup\{(Y', E)\}$ form a soft open cover of $\widetilde{X}$. Therefore
\[
(X, E)\widetilde{\subseteq}(F_{\alpha_1},E)\cup\ldots\cup(F_{\alpha_n,E})\cup(Y',E),
\]
for some $\alpha_1,\ldots,\alpha_n\in J$. Applying the previous proposition we can see that $\{(^YF_{\alpha_i},E)\}_{i=1}^n$ is a subcover of $\widetilde{Y}$. This completes the proof.
\end{proof}

\bigskip

Let $(X, \tau, E)$ be a soft topological spaces and $\mathcal{B}\subseteq\tau$. If every element of $\tau$ can be written as a union of elements of $\mathcal{B}$, then $\mathcal{B}$ is called a soft basis for the soft topology $\tau$. Each element of $\mathcal{B}$ is called a soft basis element.

We can characterize soft compact spaces in term of basis elements as follows:
\begin{theorem}
A soft topological space $(X, \tau, E)$ is soft compact if and only if there is a soft basis $\mathcal{B}$ for $\tau$ such that every cover of $\widetilde{X}$ by elements of $\mathcal{B}$ has a finite subcover.
\end{theorem}
\begin{proof}
Let $(X, \tau, E)$ be soft compact. Obviousely, $\tau$ is a soft basis for $\tau$. Therefore, every cover of $\widetilde{X}$ by elements of $\tau$ has finite subcover. Conversely, let $\{(U_\alpha, E)\}_{\alpha\in J}$ be a soft open cover of $\widetilde{X}$. We can write $(U_\alpha, E)$ as a union of basis elements, for each $\alpha\in J$. These elements form a soft open cover of $\widetilde{X}$ such as $\{(F_{\mathcal{B}}, E)\}_{\mathcal{B}\in I}$. Therefore $\widetilde{X}=(F_{\mathcal{B}_1}, E)\cup\ldots\cup(F_{\mathcal{B}_n}, E)$, for some $\mathcal{B}_1,\ldots,\mathcal{B}_n\in I$. Let $(F_{\mathcal{B}_i}, E)\widetilde{\subseteq}(U_{\alpha_i}, E)$, for each $1\leq i\leq n$. This implies that $\{(U_{\alpha_i}, E)\}_{i=1}^n$ is a finite subcover of $\widetilde{X}$. Hence, $(X, \tau, E)$ is soft compact.
\end{proof}

\begin{rem}\label{rem36}
Clearly, a soft set is not a set. Indeed, the differences between
soft topological spaces and topological spaces arise from this
fact. In a sense, when $|E|=1$, a soft set $(F,E)$ behaves
similar to a set. In fact, in this case the soft set $(F,E)$ is
the same as the set $F(e)$, where $E=\{e\}$. Therefore when
$|E|=1$, the soft topological spaces are the same as topological
spaces. Nevertheless, in this paper we will see some differences
between these two concepts when $|E|\geq 2$.
\end{rem}

\bigskip

Now, we consider a countably soft compact space constructed
around a soft topology. A soft topological space $(X, \tau, E)$ is said to be
\emph{countably soft compact} if every countable soft open cover of
$\widetilde{X}$ contains a finite subcover of $\widetilde{X}$.
Obviously, every soft compact space is countably soft compact but
the following example shows that the converse is not true in
general.
\begin{example}
We consider the (topological) space $S_{\Omega}$, the minimal
uncountable well-ordered set with order topology (see \cite{Ma}).
Let $X=S_\Omega$, $E=\{e\}$ and $\tau=\{(F,E) | F(e)$ is open in
$S_\Omega\}$. Considering Remark \ref{rem36}, the soft topological space $(X, \tau,
E)$ is countably soft compact but not soft compact.
\end{example}

\bigskip

There is a criterion for a soft space to be countable soft compact in term
of soft closed sets rather than soft open sets. First we have a
definition.

A collection $\mathcal{A}$ of soft set is said to have the
finite intersection property if for every finite sub-collection
$\{(A_1,E)\cap\ldots\cap(A_n,E)\}$ of $\mathcal{A}$, the intersection
$(A_1,E)\cap\ldots\cap(A_n,E)$ is non-null.

\begin{theorem}
A soft topological space is countably soft compact if and only if
every countable family of soft closed sets with the finite
intersection property has a nonnull intersection.
\end{theorem}
\begin{proof}
Let the soft space $(X, \tau, E)$ be countably soft compact. Let
the family $\{(F_n,E)\}_{n=1}^\infty$ of soft closed sets have
the finite intersection property. If
$\cap_{n=1}^\infty(F_n,E)=\phi_E$ by Proposition \ref{3.3},
$\{(F_n,E)'\}_{n=1}^\infty$ is a countable soft open cover of
$\widetilde{X}$. Therefore
$\widetilde{X}=(F_{n_1},E)\cup\ldots\cup(F_{n_k},E)$, for some
$n_1,\ldots,n_k\in N$. Now, De Morgan's lows and Proposition \ref{3.3}
imply that $(F_{n_1},E)\cap\ldots\cap(F_{n_k},E)=\phi_E$. This is
a contradiction. Conversely, Let $\{(F_n,E)\}_{n=1}^\infty$ be a
countable soft open cover of $\widetilde{X}$ without any
subcover. Then $\{(F_n,E)'\}_{n=1}^\infty$ is a family of soft
closed sets over $X$ such that
$\cap_{n=1}^\infty(F_n,E)'=\phi_E$. Let $n_1,\ldots,n_k$ be
arbitrary positive integers. If
$(F_{n_1},E)'\cap\ldots\cap(F_{n_k},E)'=\phi_E$ then $\widetilde{X}=(F_{n_1},E)\cup\ldots\cup(F_{n_k},E)$, that is impossible. Therefore $(F_{n_1},E)'\cap\ldots\cap(F_{n_k},E)'\neq\phi_E$, for each
$n_1,\ldots,n_k\in N$. This shows that
$\{(F_n,E)'\}_{n=1}^\infty$ have the finite intersection
property. Therefore $\cap_{n=1}^\infty(F_n,E)'\neq\phi_E$. This
is a contradiction.
\end{proof}
An immediate result of previous theorem is the following.
\begin{cor}
A soft space $(X, \tau, E)$ is countably soft compact if and only
if every nested sequence
$(F_1,E)\widetilde{\supseteq}(F_2,E)\widetilde{\supseteq}\ldots$
of nonnull soft closed sets over $X$ has a nonnull intersection.
\end{cor}
\begin{proof}
Let $(X, \tau, E)$ is countably soft compact. The collection
$\{(F_n,E)\}_{n=1}^\infty$ have the finite intersection property.
Therefore $\cap_{n=1}^\infty(F_n,E)\neq\phi_E$. Conversely, let
$\{(C_n,E)\}_{n=1}^\infty$ be a collection of soft closed sets
with the finite intersection property. By Proposition \ref{prop21},  we
construct nested sequence
$(F_1,E)\widetilde{\supseteq}(F_2,E)\widetilde{\supseteq}\ldots$
of nonnull soft closed sets by setting
$(F_n,E)=(C_1,E)\cap\ldots\cap (C_n,E)$, for each positive
integer $n$. By the hypothesis $\cap_{n=1}^\infty(F_n,E)=\cap_{n=1}^\infty(C_n,E)\neq\phi_E$. Now,  Theorem 3.8 implies that $(X, \tau, E)$ is countably soft compact.
\end{proof}

\section{Soft Separation Axioms}
In this section,  we will study  some soft separation
axioms that have studied in \cite{WKM, SN}. First, we recall the
definitions.

A soft topological space $(X, \tau, E)$ over $X$ is called a soft
$T_0$-space if for each pair of distinct points, at least one has
neighborhood not containing the other, and a soft $T_1$-space if for each pair of distinct points, each one has a neighborhood not containing the other. Also, the soft space $(X, \tau, E)$ is said to be soft $T_2$- space (or soft Hausdorff) if for each pair $x,y$ of distinct points of $X$,
there exist disjoint soft open sets containing $x$ and $y$, respectively.

Obviously, every soft $T_i$-space $(i=1,2)$ is a soft
$T_{i-1}$-space. But by Remark 3.6 and general topology the
converse is not true. In \cite{SN},  the authors have shown that if
$(x, E)$ is a soft closed set in soft set $(X, \tau, E)$, for all
$x\in X$, then $(X, \tau, E)$ is soft $T_1$, but the converse does
not hold in general.

The soft space $(X, \tau, E)$ over $X$ is called soft regular if
for each soft closed set $(G,E)$ and $x\in X$ such that $x \notin (G,E)$
there exist soft open sets $(F_1,E)$ and $(F_2,E)$ such that
$x\in (F_1,E),(G,E)\widetilde{\subseteq}(F_2,E)$ and
 $(F_1,E)\cap (F_2,E)=\phi_E$. The soft space $(X, \tau, E)$ is
 said to be soft $T_3$-space if it is soft regular and soft
 $T_1$-space.

Before proceeding,  we introduce the concept of soft closure of a
soft set (see \cite{HA}). Let $(X, \tau, E)$ be a soft topological
space and $(F,E)$ be a soft set over $X$. Then the soft closure of
$(F,E)$, denoted by $(\overline{F,E})$,  is the intersection of all
soft closed super sets of $(F,E)$. First, we prove the following.
\begin{lemma}\label{lem41}
Let $(X, \tau, E)$ be a soft topological space and
$(F,E)$ be a soft set over $X$. If $x\in (\overline{F,E})$, then
every soft open set $(G,E)$ containing $x$ intersects $(F,E)$.
\end{lemma}
\begin{proof}
Let $x\in (\overline{F,E})$. Let there is a soft open set $(G,E)$
containing $x$ such that $(F,E)\cap (G,E)=\phi_E$. By Proposition \ref{3.2},
we have $(F,E)\widetilde{\subseteq}(G,E)'$. Therefore
$(\overline{F,E})\widetilde{\subseteq}(G,E)'$. Hence $x\in
(G,E)\cap(G,E)'$. This is a contradiction. Therefore
$(F,E)\cap(G,E)\neq \phi_E$, for each soft open set $(G,E)$
containing $x$.
\end{proof}
The following example shows that the converse of Lemma \ref{lem41} is not
true.
\begin{example}
Suppose that the following  sets are given:  $X=\{h_1,h_2,h_3\}$, $E=\{e_1,e_2\}$ and $\tau=\{\phi_E,\widetilde{X},(F_1,E),(F_2,E),\ldots,(F_{30},E)\}$
where $F_1,F_2,\ldots,F_{30}$ are given in Example 9 of \cite{SN}.
Then $(X,\tau)$ is a soft topological space over $X$. We consider
the soft set $(F_{25},E)$, where
\[
F_{25}(e_1)=\{h_2\},\ \ F_{25}(e_2)=X.
\]
It is easy to see that the following hold
\[
(\overline{F_{25},E})=(F_{25},E),\  h_1\notin (\overline{F_2,E}).
\]
But for every soft open set $(F,E)$ containing $h_1$,  we have
$(F,E)\cap (F_{25},E)\neq\phi_E$.
\end{example}

\bigskip

\begin{proposition}\label{prop43}
Let $(X, \tau, E)$ be a soft regular space. Then, for each point
$x$ of $X$ and a soft open set  $(F,E)$ containing $x$, there is a soft open set $(G,E)$ containing $x$ such that
$(\overline{G,E})\widetilde{\subseteq}(F,E)$.
\end{proposition}
\begin{proof}
$(F,E)'$ is a soft closed set not containing $x$. Therefore,
there exist soft open sets $(G,E)$ and $(H,E)$ such that $x\in
(G,E)$, $(F,E)'\widetilde{\subseteq}(H,E)$ and $(G,E)\cap
(H,E)=\phi_E$. Proposition \ref{3.2} implies that
$(G,E)\widetilde{\subseteq}(H,E)'$. Therefore
$(\overline{G,E})\widetilde{\subseteq}(H,E)'\widetilde{\subseteq}((F,E)')'=(F,E)$.
\end{proof}

The following example shows that the converse of Proposition \ref{prop43} does not hold in general.

\begin{example}\label{ex44}
Let $X=\{h\}$, $E=\{e_1,e_2\}$ and $\tau=\{\phi_E, \widetilde{X},
(F_1,E),(F_2,E) \}$, where
\[
F_1(e_1)=\{h\},\ F_1(e_2)=\emptyset\ \ \&\ \ F_2(e_1)=\emptyset,
F_2(e_2)=\{h\}
\]
\end{example}

\bigskip

It is easy to see that $(X, \tau, E)$ is not soft regular
Nevertheless, for $h\in X$ and soft open set $\widetilde{X}$
containing $h, \widetilde{X}$ itself is a soft open set
containing $h$ such that $h\in
\overline{{\widetilde{X}}}\widetilde{\subseteq} \widetilde{X}$.

Now, we exhibit a necessary and sufficient condition for a soft
space to be a soft regular space.
\begin{theorem}
A soft space $(X, \tau, E)$ is soft regular if and only if for
each $x\in X$ and soft closed set $(F,E)$ not containing $x$,
there is a soft a open set $(G,E)$ containing $x$ such that
$(\overline{G,E})\cap(F,E)=\phi_E$.
\end{theorem}
\begin{proof}
Let $(X, \tau, E)$ be soft regular. There exist soft open sets
$(G,E)$ and $(H,E)$ such that $x\in(G,E)$,
$(F,E)\widetilde{\subseteq}(H,E)$ and $(G,E)\cap(H,E)=\phi_E$. Then
$(G,E)\widetilde{\subseteq}(H,E)'\widetilde{\subseteq}(F,E)'$.
This implies that
$(\overline{G,E})\widetilde{\subseteq}(H,E)'\widetilde{\subseteq}(F,E)'$.
Therefore $(\overline{G,E})\cap(F,E)=\phi_E$.

Conversely, Proposition \ref{3.2} implies that
$(F,E)\widetilde{\subseteq}(\overline{G,E})'$. Therefore there is
a soft open set $(\overline{G,E})'$ containing $(F,E)$ such that
$(G,E)\cap(\overline{G,E})'=\phi_E$. This completes the proof.
\end{proof}
A soft space topological space $(X, \tau, E)$ is said to be soft
normal if for each soft closed sets $(F,E)$ and $(G,E)$ over $X$
with null intersection there exist soft open sets $(F_1,E)$ and
$(F_2,E)$ containing $(F,E)$ and $(G,E)$ respectively, such that
$(F_1,E)\cap(F_2,E)=\phi_E$. Also, a soft topological space $(X,
\tau, E)$ is said to be a soft $T_4$-space if it is soft normal
and soft $T_1$-space.
\begin{theorem}
Let $(X, \tau, E)$ be a soft space. Let for each soft closed set
$(F,E)$ and soft open set $(G,E)$ containing $(F,E)$ there is a
soft open set $(H,E)$ containing $(F,E)$ such that
$(\overline{H,E})\widetilde{\subseteq}(G,E)$. Then $(X, \tau, E)$
is soft normal.
\end{theorem}
\begin{proof}
For each soft closed sets $(F,E)$ and $(I,E)$ with null
intersection $(I,E)'$ is a soft open set containing $(F,E)$.
Therefore there is a soft open set $(H,E)$ containing $(F,E)$ such that
$(\overline{H,E})\widetilde{\subseteq}(I,E)'$. By Proposition
\ref{3.2}, $(I,E)\widetilde{\subseteq}(\overline{H,E})'$. Since
$(H,E)\widetilde{\subseteq}((\overline{H,E})')'$, we have
$(H,E)\cap(\overline{H,E})'=\Phi_E$. Hence $(X, \tau, E)$ is
soft normal.
\end{proof}

\bigskip

There is an obvious question to ask at this point. Is a soft
$T_4$-space a soft $T_3$-space? The soft space $(X, \tau, E)$ in
Example \ref{ex44},  shows that the answer is ''NO''. In fact it is easy
to see that $(X, \tau, E)$ is a soft $T_4$-space and not a soft
$T_3$-space.

\bigskip
\begin{rem}
In Theorem 3.17 of \cite{WKM}, the following is proved:\\\\
Theorem. (\cite{WKM}) Let $(X, \tau, E)$ be a soft topological space over $X$ and $x\in X$. Then the following are equivalent:
\begin{description}
  \item[(1)] $(X, \tau, E)$ is a soft regular space;
  \item[(2)] For each soft closed set $(G,E)$ such that
$(x,E)\cap(G,E)=\phi_E$.
\end{description}
There exit soft two open sets $(F_1,E)$ and $(F_2,E)$ such that $(x,E)\widetilde{\subseteq}(F_1,E)$,
$(G,E)\widetilde{\subseteq}(F_2,E)$ and $(F_1,E)\cap(F_2,E)=\phi_E$.

\bigskip
\noindent
By Example \ref{ex44},  we can see that this theorem is incorrect. In fact
the soft space $(X, \tau, E)$ in this example satisfies in (2),
but it is not soft regular. We note that $(x,E)\cap(G,E)=\phi_E$
is not equivalent to $x\notin (G,E)$. But
$(x,E)\widetilde{\nsubseteq}(G,E)$ is. Therefore, we must replace
the condition $(x,E)\widetilde{\nsubseteq}(G,E)$ instead of
$(x,E)\cap(G,E)=\phi_E$ in Theorem 3.17 of \cite{WKM}.
\end{rem}

\bigskip

\begin{rem}
In Theorem 3.25 of \cite{WKM}, the following is proved:\\\\
Theorem. (\cite{WKM}) Let $(X, \tau, E)$ be a soft topological
space over $X$. If $(X, \tau, E)$ is a soft normal space and if
$(x,E)$ is a soft closed set for each $x\in X$, then $(X, \tau,
E)$ is a soft $T_3$-space.

This theorem is incorrect. The soft space $(X, \tau, E)$ in
Example \ref{ex44} satisfies in the conditions of the theorem, but it is
not a soft $T_3$-space.
\end{rem}

There are some familiar results on the applications of
compactness in separation axioms in General Topology such as: \emph{Every compact Hausdorff space is normal}. But it is not true about soft topological spaces. Consider the following example.
\begin{example}
Let  $X=\{h\}$, $E=\{e_i\}_{i=1}^5$ and $\tau=\{\phi_E,
\widetilde{X}, (F_1,E),(F_2,E), (F_3,E) \}$, where
\[
F_1(e_1)=\emptyset,\ F_1(e_2)=X,\  F_1(e_3)=\emptyset,\
F_1(e_4)=X,\  F_1(e_5)=\emptyset;
\]
\[
F_2(e_1)=X,\ F_2(e_2)=X,\  F_2(e_3)=X,\ F_2(e_4)=\emptyset,\
F_2(e_5)=X;
\]
\[
F_3(e_1)=\emptyset,\ F_3(e_2)=X,\  F_3(e_3)=\emptyset,\
F_3(e_4)=\emptyset,\  F_3(e_5)=\emptyset.
\]
It is easy to see that $(X, \tau, E)$ is not soft normal.
Nevertheless, it is soft compact.
\end{example}

It is remarkable that every compact Hausdorff space is not normal, even if we consider $(X, \tau, E)$ as a
soft regular space. Indeed, the Example \ref{ex44} is a counterexample.


\bigskip

\noindent
Esmaeil Peyghan and Babak Samadi\\
Department of Mathematics, Faculty  of Science\\
Arak University\\
Arak 38156-8-8349,  Iran\\
Email: epeyghan@gmail.com
\bigskip

\noindent
Akbar Tayebi\\
Department of Mathematics, Faculty  of Science\\
University of Qom \\
Qom. Iran\\
Email:\ akbar.tayebi@gmail.com

\end{document}